\documentclass[12pt]{elsarticle}
\usepackage{amsmath}
\usepackage{amssymb}
\usepackage{fancybox}
\usepackage{eucal}
\usepackage{amsthm}
\usepackage{amscd}
\usepackage{hyperref}
\usepackage{wasysym}
\usepackage{url}

\usepackage{amssymb,}
\usepackage[]{amsmath, amsthm, amsfonts,graphicx, amscd,}
\usepackage[all,cmtip]{xy}
\input amssym.def \input amssym
\begin{document}

\newtheorem {thm}{Theorem}[section]
\newtheorem{corr}[thm]{Corollary}
\newtheorem {alg}[thm]{Algorithm}
\newtheorem*{thmstar}{Theorem}
\newtheorem{prop}[thm]{Proposition}
\newtheorem*{propstar}{Proposition}
\newtheorem {lem}[thm]{Lemma}
\newtheorem*{lemstar}{Lemma}
\newtheorem{conj}[thm]{Conjecture}
\newtheorem{ques}[thm]{Question}
\newtheorem*{conjstar}{Conjecture}
\theoremstyle{remark}
\newtheorem{rem}[thm]{Remark}
\newtheorem{np*}{Non-Proof}
\newtheorem*{remstar}{Remark}
\theoremstyle{definition}
\newtheorem{defn}[thm]{Definition}
\newtheorem*{defnstar}{Definition}
\newtheorem{exam}[thm]{Example}
\newtheorem*{examstar}{Example}
\newcommand{\pd}[2]{\frac{\partial #1}{\partial #2}}
\newcommand{\pdtwo}[2]{\frac{\partial^2 #1}{\partial #2^2}}
\def\Ind{\setbox0=\hbox{$x$}\kern\wd0\hbox to 0pt{\hss$\mid$\hss} \lower.9\ht0\hbox to 0pt{\hss$\smile$\hss}\kern\wd0}
\def\Notind{\setbox0=\hbox{$x$}\kern\wd0\hbox to 0pt{\mathchardef \nn=12854\hss$\nn$\kern1.4\wd0\hss}\hbox to 0pt{\hss$\mid$\hss}\lower.9\ht0 \hbox to 0pt{\hss$\smile$\hss}\kern\wd0}
\def\ind{\mathop{\mathpalette\Ind{}}}
\def\nind{\mathop{\mathpalette\Notind{}}} 
\newcommand{\m}{\mathbb }
\newcommand{\mc}{\mathcal }
\newcommand{\mf}{\mathfrak }

\title{Indecomposability for differential algebraic groups}
\author{James Freitag}

\address{freitag@math.berkeley.edu \\
Department of Mathematics\\
University of California, Berkeley\\
970 Evans Hall\\
Berkeley, CA 94720-3840 }
\begin{abstract}  We study a notion of indecomposability in differential algebraic groups which is inspired by both model theory and differential algebra. After establishing some basic definitions and results, we prove an indecomposability theorem for differential algebraic groups. The theorem establishes a sufficient criterion for the subgroup of a differential algebraic group generated by an infinite family of subvarieties to be a differential algebraic subgroup. This theorem is used for various definability results. For instance, we show every noncommutative almost simple differential algebraic group is perfect, solving a problem of Cassidy and Singer. We also establish numerous bounds on Kolchin polynomials, some of which seem to be of a nature not previously considered in differential algebraic geometry; in particular, we establish bounds on the Kolchin polynomial of the generators of the differential field of definition of a differential algebraic variety. 
\end{abstract}

\maketitle

The indecomposability theorem of Zilber, generalized to the setting of weakly categorical groups \cite{Zilberindecomp} the well-known theorem of algebraic group theory, which states that a family of irreducible subvarieties passing through the identity element of an algebraic group generates an algebraic subgroup \cite{Rosenlitch1}. The theorem is a powerful tool for definability results in groups of finite Morley rank. Zilber's theorem was generalized to the superstable (possibly infinite rank) setting by Berline and Lascar \cite{BerlineLascar}. 

Because the theory of differentially closed fields of characteristic zero with $m$ commuting derivations, denoted by $DCF_{0,m}$, is $\omega$-stable, hence superstable, their generalization of Zilber's theorem holds for differential algebraic groups, which are the definable groups in $DCF_{0,m}$. However, when $m>1$, it is difficult to show that a differential algebraic group satisfies the connectivity hypothesis of the Berline-Lascar indecomposability theorem \cite{BerlineLascar}. In the case of superstable groups, the connectivity hypotheses are phrased in terms of Lascar rank. In the theory of differential algebraic groups, another notion of dimension arises more often in applications; the Kolchin polynomial is a numerical polynomial which tracks the growth of transcendence degree of a generic point of the group under application of the derivations. There is no known lower bound for Lascar rank in terms of the known differential birational invariants of the Kolchin polynomial \cite{SuerThesis}. The lack of control of Lascar rank in terms of the Kolchin polynomial affects both the hypothesis and conclusion of the differential algebraic case of the Berline-Lascar indecomposablity theorem. With that in mind, we prove an idecomposability theorem for differential algebraic groups in which both the hypothesis and the conclusion are purely differential algebraic in nature, although the proof contains model-theoretic concepts and techniques. 

Cassidy and Singer \cite{CassidySinger}, introduced \emph{almost simple} and \emph{strongly connected} differential algebraic groups, and these notions play prominent roles here. 
Every differential algebraic group $G$ has a strongly connected component $H$, which is a characteristic subgroup such that $G/H$ is much smaller than $G$ as measured by the degree of the Kolchin polynomial, which Kolchin called the \emph{differential type} of $G$. In the classical theory of differential equations, this degree is the number of independent variables on which the general solution of a system of differential equations depends. A differential algebraic group is \emph{almost simple} if every proper definable normal subgroup of $G$ has smaller differential type than $G$. Almost simple differential algebraic groups are strongly connected. Using a Jordan-H\"older style theorem of \cite{CassidySinger}, one sees that every differential algebraic group has a subnormal series in which the successive quotients are almost simple. The definability of the quotients follows from the fact that $DCF_{0,m}$ \emph{eliminates imaginaries}. Quotient structures were also investigated by Kolchin \cite{KolchinDAG}. Further, the quotients which appear are unique up to a suitable notion of isogeny. As is clear in the later sections of \cite{CassidySinger}, analyzing the noncommutative almost simple groups would be aided by knowing that the groups were perfect. This follows from a corollary of our main indecomposability theorem - derived subgroups of strongly connected groups are closed in the Kolchin topology. The analysis of strongly connected groups of small typical differential dimension can be carried out \cite{JIndecomp2} along the same lines as the analysis of groups of small Morley rank \cite{CherlinSmall}. Perfection of the noncommutative almost simple groups also seems to be a necessary requirement for studying the groups via algebraic $K$-theory along the lines of \cite{CherlinAltinel}. We do not pursue that route in this paper; see \cite{CassidySinger} for a discussion of this topic. 

We also discuss open problems which could connect our results to those of Berline and Lascar \cite{BerlineLascar}. Generalizing this work to the difference-differential setting (or even more general settings) is of interest, but is not covered here. Also, though there are well-developed theories of numerical polynomials in these more general settings \cite{PolyD}; work along the lines of \cite{CassidySinger} (in those settings) would seem to be a prerequisite for proving results like those in this paper. The main original motivation for this work was the analysis of group theoretic properties of almost simple (and more generally strongly connected) differential algebraic groups. As mentioned above, this sort of analysis is one of the first steps in any attempt to classify almost simple differential algebraic groups. For further discussion of these issues, see \cite{CassidySinger}. 

In addition to the group theoretic developments in this paper, Theorem \ref{first} is likely to be of independent interest in differential algebraic geometry. Start with a differential field $k$ and a tuple $\bar a$ in a differential field extension.

Roughly speaking, the result says that if we extend $k$ to a differential field $K_1$ in such a way that the Kolchin polynomial of $\bar a$ over $k$ differs from the Kolchin polynomial of $\bar a$ over $K_1$ in a specific bounded manner, then there is an intermediate field extension, $K$ such that the Kolchin polynomial of $\bar a$ over $K$ is equal to the Kolchin polynomial of $\bar a$ over $K_1$ and $K$ is generated as a differential field extension over $k$ by some tuple which has a Kolchin polynomial that is similarly bounded. So, the result gives bounds on Kolchin polynomials of the generators of relative fields of definition (or simply fields of definition, assuming $k=\m Q$). In proving Theorem \ref{first}, we found need for the analogue of the Lascar inequality for Kolchin polynomials. Let $A \subseteq K$. We remind the reader that the Lascar inequality says:
$$RU ( \bar a / A \langle b \rangle ) + RU ( \bar b /A) \leq RU ((\bar a,  \bar b) /A ) \leq RU ( \bar a / A \langle b \rangle ) \oplus RU ( \bar b /A).$$ 
The sum of two ordinals, denoted by $\alpha + \beta$, is the order type of $\alpha \cup \beta$ with the order given by letting every element of $\alpha$ be less than any element of $\beta$. 
We should mention that every ordinal may be written uniquely in the form 
$$\omega ^ {\alpha _1} a_1 + \omega ^{\alpha _2} a_2 + \ldots + \omega ^{ \alpha _k} a_k,$$ 
where $\alpha_i$ are ordinals and $a_i$ are \emph{positive integers}. For ordinals written in this \emph{Cantor normal form}, the operation $\oplus$ is simple; the sum of two such ordinals works as if summing polynomials in $\omega$. We do not know what the analogue of the right inequality (the upper bound) ought to be in the context of Kolchin polynomials. However, the left inequality (the lower bound) was established in a preliminary version of this paper. The referee pointed out that this result was embedded in a paper of William Sit \cite{SitPoly}. The precise statement of Sit's Lemma is recorded here as Lemma \ref{Sitlemma}. 

The paper is organized as follows. Section 1 sets up the basic definitions and background from model theory and differential algebra, providing references when this is not feasible. Section 2 covers stabilizers of types (in the sense of model theory, see the definitions of section \ref{defnn}) in differential algebraic groups. The analysis of indiscernible sequences is a major tool in model theory. We use indiscernible sequences to calculate bounds on the Kolchin polynomials of stabilizers of certain types (in the sense of model theory) in differential algebraic groups. In particular, in this section we prove several new results in differential algebraic geometry, some having nothing in particular to do with groups. For instance, Theorem \ref{first} provides a bound on the Kolchin polynomial of the smallest field of definition of a variety. Section 3 gives the main theorem, a definability result for differential algebraic groups, in which we prove that a certain a priori abstract subgroup of a differential algebraic group is, in fact, a definable subgroup. Section 4 and section 5 give applications of the main theorem and show some specific examples. Section 6 discusses some open problems and possible generalizations of the work. Generalizations of the conditions of the theorem in the setting of differential fields are considered. We also discuss generalization of the setting itself, via adding more general operators to the fields.

\section{Definitions and notation}\label{defnn}
A differential field is a field $K$ together with $m$ commuting derivations, $\delta_i : K \rightarrow K$. Let $(\delta_1 ^{i_1} \dots \delta_m ^{i_m} x_j) _{1 \leq j \leq n , i_j \in \m N}$ be a family of indeterminates over $K$. The finitely generated differential polynomial algebra $R = K \{ x_1, \dots , x_n \}$ over $K$ equals the countably generated polynomial algebra $K [ (\delta_1 ^{i_1} \dots \delta_m ^{i_m} x_j) _{1 \leq j \leq n , i_j \in \m N} ],$ with the obvious extension of the derivation operators to $R$. The generators are called \emph{differential indeterminates,} and the elements of $R$ are called \emph{differential polynomials.} 
Let $I$ be an index set. A system of partial differential algebraic equations is a collection $(P_i=0)_{i \in I}$, where $P_i$ is a differential polynomial. 
 A solution to the system is a tuple $\bar  a \in L^n$ so that $P_i (\bar a )=0$ for $i \in I$ where $L$ is a differential field extension of $K.$ We use the notation $K \langle a \rangle$ for the differential field generated by $a$ over $K$. 

The model theory of partial differential fields of characteristic zero with finitely many commuting derivations was developed in \cite{McGrail}. In this setting, there is a model companion, which we denote $DCF_{0,m},$ which is characterized by the property that finite systems of differential polynomial equations over a model of $DCF_{0,m}$ which have a solution in some differential field extension already have a solution in the model. Writing this characterization via first order axioms is slightly more complicated \cite{McGrail}. 
For a reference in differential algebra, we suggest \cite{KolchinDAAG} and \cite{MMP}. Throughout this paper, $K \models DCF_{0,m}$ (this is standard model theoretic notation for ``$K $ satisfies the axiom scheme $DCF_{0,m}$") will be a field over which our sets are defined and $\mc U$ is a very saturated model of $DCF_{0,m}.$ This assumption amounts to saying that $\mc U$ is a universal domain in the sense of differential algebraic geometry. Thinking of varieties over $K$ as synonymous with their $\mc U$-points is a convenient (though not strictly necessary) approach, which we will take in this paper. We will usually be taking our varieties to be defined over $K$. This does not amount to any restriction (since $K$ is not fixed), but it is, again, a convenient approach for our purposes. All tuples in differential field extensions of $K$ can be assumed to come from $\mc U$. We will occasionally use $k$ for a differential field, which we \emph{do not} assume is differentially closed. 

As in the case of algebraically closed fields, $DCF_{0,m}$ has quantifier elimination, which gives a bijective correspondence between irreducible differential subvarieties of $\m A^n$, complete $n$-types, and prime differential ideals contained in $K \{ x_1 , \ldots , x_n \}$. We will make this explicit and fix notation next. Given an $n$-type over $K$, that is, $p \in S_n(K)$, we have a corresponding differential ideal via
$$p \mapsto I_p = \{ f \in K\{x_1, \ldots , x_n \} \, | \, f(\bar x )=0 \in p \}.$$
Throughout the paper, $S_n(K) $ denotes the space of complete $n$-types over $K$. We also denote the union of these spaces by $S(K) := \cup _{n \in \m N} S_n (K)$. 
Of course, the corresponding variety is simply the zero set of $I_p.$ We will use this correspondence implicitly throughout, including in the notation of \emph{Kolchin polynomials}, which we will define later. There is also a natural inclusion-reversing bijection between prime differential ideals in $K \{ x_1 , \ldots , x_n \}$ and differential subvarieties of $\m A^n$ over $K$. For a tuple $\bar a \in \mc U$, the type of $\bar a$ over $K$, denoted $tp(\bar a /K)$ is the collection of all first order formulas $\phi (\bar x)$ with parameters in $K$ which are satisfied by $\bar a$. 

We use the tools of stability theory throughout the paper. We will quickly discuss the basic definitions, but for the reader unfamiliar with the language of model theory this may be insufficient, so the reader is referred to \cite{Marker} and \cite{GST}. We work over some differential field $k \subseteq K$ (as per our above notation). Model theorists have developed a very general notion of \emph{independence} of two sets over a base field. This notion applies in great generality, and it is not feasible to describe its entire development here. 

We will, however, explain what independence means in the setting of differential algebraic geometry. Two tuples $\bar a$ and $\bar b$ are said to be \emph{independent over $k$} if $k \langle \bar a \rangle $ and $k\langle \bar b \rangle$ are linearly disjoint over $k^{alg}$. So, in the setting of differential algebra, the general model-theoretic notion has a familiar concrete manifestation. 
Alternatively, for the reader more familiar with model theory, this is equivalent to requiring that $RU( tp( \bar a/k)) = RU (tp( \bar a/ k\langle \bar b \rangle )).$ For the reader more familiar with differential algebra, this is equivalent to requiring equality of the Kolchin polynomials (which we will define in the coming paragraphs) $\chi_{\bar a/k}(t)=\chi_{\bar a/k\langle \bar b \rangle}(t).$ We use standard model-theoretic notation for independence, $\bar a \ind _k \bar b,$ reads ``$\bar a$ is independent from $\bar b$ over $k.$" When $\bar a$ is not independent from $\bar b$ over $k,$ we write $\bar a \nind _k \bar b.$ Sometimes in this case we say ``$\bar a$ forks with $\bar b$ over $k.$"  One may also replace $\bar a$ and $\bar b$ with infinite sets, and in that case, the sets are independent if and only if any finite tuples selected from them are independent.
Like many notions in model theory, independence is invariant up to interdefinability, so we could assume all of the sets which appear to be \emph{definably closed}. In differential fields, $\bar b$ is in the definable closure of $\bar a$ over $k$ if $\bar b \in k \langle \bar a \rangle.$ If $k$ is not specified, then one can assume $k= \m Q$. Two tuples are interdefinable if they are in each other's's definable closure, that is $\bar a $ and $\bar b$ are interdefinable when $\m Q \langle \bar a \rangle = \m Q \langle \bar b \rangle. $ 

Now, suppose that $A \subset B \subset K$ are differential fields. For a type $q \in S_n(B)$ extending a type $p \in S_n(A),$ we say $q$ is a \emph{nonforking extension} of $p$ if $c \models q$ implies $c \ind_A B.$ For the reader more familiar with differential algebra, this is equivalent to requiring equality of the Kolchin polynomials over the types, $\chi_p(t) =\chi_q(t).$ The Kolchin polynomial will be below. 

As we have explained above, model theorists often consider \emph{Lascar rank} on types. The Lascar rank of a type $p \in S(A)$, denoted $RU(p)$ is the ordinal (or infinity): 
$$ sup \{RU(q)+1 \, : \, \text{there is } B, \, A \subset B \subset \mc U, \, q \in S(B), \, p \subset q, \, q \text{ forks over } A \}$$ 
In an arbitrary theory, the Lascar rank might be infinite (larger than every ordinal), but differentially closed fields have the model-theoretic property of \emph{superstability}; the Lascar rank of a type over a differential field is less than $\omega ^ {m+1}$ where $m$ is the number of derivations. Calculating Lascar rank in differential fields is quite nontrivial; we discuss this point further later in the paper. 

Next, we develop Kolchin polynomials, which are essential to the rest of the paper. The main tool in the Berline-Lascar extension to superstable groups of Zilber's indecomposability theorem is the deepening of the topological property of connectivity of algebraic groups to a non-topological concept,, $\alpha$-connectivity, which is defined using the Lascar rank. Write the Lascar rank of a superstable group $G$ in Cantor normal form and suppose that the leading term is $\omega ^ \alpha d_ \alpha$. $G$ is $\alpha$-connected if for all proper intersections of definable subgroups, $H$, the Lascar rank of $G/H$ is at least $\omega ^ \alpha$. Independently, Cassidy and Singer deepened the topological notion of connectivity of algebraic groups when they defined a differential algebraic group of differential type $\tau$ to be \emph{strongly connected} if the differential type of $G/H$ is also $\tau$ for all proper differential algebraic subgroups $H$ of $G$. In the Cassidy-Singer formulation, there is no need to consider infinite intersections of definable (i.e., differential algebraic) subgroups; Notherianity of the Kolchin topology guarantees that any such infinite intersection will actually be given by a finite (sub-) intersection. Since differential type is defined in terms of the Kolchin polynomial, we now turn to the definition of this important algebraic notion of dimension. 

 Let $\Theta$ be the free commutative monoid generated by $\Delta = \{\delta_, \ldots , \delta _m \}.$   For $\theta \in \Theta,$ if $\theta=\delta_1^{\alpha_1} \ldots \delta_m^{\alpha_m},$ then $ord(\theta)=\alpha_1+ \ldots + \ldots + \alpha_m.$ The order gives a grading on the monoid $\Theta$. We let $\Theta (s)=\{ \theta \in \Theta: \, ord(\theta) \leq s \}.$ We should note that classical differential algebra often considers gradings which induce total orders on $\Theta.$ These gradings play a key role in axiomatizing differentially closed fields. Alternative orderings are also frequently considered. We will, however, not need to consider such fine gradings in our analysis. 

\begin{thm}\label{combinatorialprep}  \cite{KolchinDAAG}
Let $\eta = (\eta_1 , \ldots \eta_n ) $ be a finite family of elements in some extension of $k.$ There is a numerical polynomial $\chi_{\eta /k} (t)$ with the following properties. 

\begin{enumerate} 
\item For sufficiently large $t \in \m N,$ the transcendence degree of $k ((\theta \eta_j)_{\theta \in \Theta(t), \, 1 \leq j \leq n })$ over $k$ is equal to $\chi_{\eta /k} (t).$ 
\item $deg (\chi_{\eta /k }(t)) \leq m$  
\item One can write $$\chi_{\eta / k }(t) = \sum_{0 \leq i \leq m} a_i \binom{t+i}{i}$$ In this case, $a_m$ is the differential transcendence degree of $k \langle \eta \rangle$ over $k.$ 
\end{enumerate}
\end{thm}

When $ p \in S(k),$ and $ \bar a \models p$, we define $\chi_p (t) := \chi _{\bar a/k} (t).$ The Kolchin polynomial of a differential variety is a birational invariant, that is an invariant of the field of $k$-rational functions $k(\bar a)$ of $\bar a$. It is not a differential birational invariant, that is, it depends on the particular chosen generators of $k \langle \bar a \rangle.$ The leading coefficient and the degree of the Kolchin polynomial are differential birational invariants. We call the degree the \emph{differential type} or $\Delta$-type of $V$. Recalling the bijective correspondence between types, varieties and tuples in differential field extensions detailed above, let $V$ be the differential variety corresponding to the defining prime ideal of $\bar a$ over $k$. We will use the notation $\tau(V/k):=\tau (\bar a/k)$ for the the differential type. Noting the above correspondence between tuples in field extensions (realizations of types) and varieties, we will occasionally write $\tau(p)$ or $\tau(\bar a).$ Again, when we wish to emphasize the base set, we will write $\tau(\bar a/k).$ 

Since types come over a specified set, this is never necessary for types, however, we often wish to consider restrictions of types to smaller differential fields. Suppose $k_1 \leq k_2$ are differential fields and $p \in S (k_2)$. Then we write $p|_{k_1}$ for the restriction of $p$ to $k_1$. Taking $\bar a \models p$, we will often consider $\chi _{p|_{k_1}} (t) = \chi _{\bar a /k_1} (t)$, which is of course always greater than or equal to $\chi _p (t) = \chi _{\bar  a/k_2} (t).$ 

The leading coefficient of the Kolchin polynomial is called the \emph{typical differential dimension} or the typical $\Delta$-dimension. We will also write $a_\tau (\bar a)$ for a tuple of elements $a$ in a field extension. We will write $a_\tau(V)$ for the typical differential dimension of a variety $V.$ As above, we write $a_\tau (p)$ for a type $p.$ Similarly, we write $a_\tau (\bar a/k)$ and $a_\tau (V/k)$ when we wish to emphasize that the calculation is being done over $k.$  

The following is elementary to prove, see \cite{MPSarcs2008}. 
\begin{lem}\label{quot} For $\bar a, \bar b$ in a field extension of $k.$ $$\tau(\bar a,\bar b /k)= max \{ \tau (\bar a/k) , \tau(tp(\bar a/ k \langle \bar b \rangle)) \}.$$
If $\tau(\bar a) =\tau(tp(\bar b/ k \langle \bar a \rangle)),$ then 
$$a_\tau ({ (\bar a, \bar b )/k}) = a_\tau ({\bar a/k}) + a_\tau ({ \bar b/ k \langle \bar a \rangle })$$
If $\tau(\bar a /k) >\tau(\bar b/k \langle \bar a \rangle ),$ then
$$a_\tau ({ (\bar a,\bar b)/k}) = a_\tau ({\bar a /k})$$
\end{lem}

For additional results on the properties of differential type and typical differential dimension, see \cite{CassidySinger} and \cite{KolchinDAAG}. 

There are several natural perspectives from which one might consider the groups in this category. 

\begin{defn}  Let $X$ be a differential algebraic variety over $k$ (see \cite{KolchinDAG} for the definition). Let $g: X \times X \rightarrow X$ be a group operation and a $\Delta$-$k$-morphism, that is, a map which is locally differential rational over $k$. In this case, $X$ is called a \emph{$\Delta$-$k$-algebraic group.}
\end{defn} 

\begin{defn} Let $X \subseteq \mc U^n$ be a definable set and let $g: X \times X \rightarrow X$ be a group operation whose graph is a definable set. In this case, $X$ is called a \emph{$\Delta$-$k$-definable group.}
\end{defn} 

By a theorem of Pillay \cite{PillayDAG}, the two categories are equivalent. 

\begin{rem} For the reader not accustomed to the model theoretic approach, the previous definition seems problematic, since the underlying set on which any definable group is defined is a definable subset of affine space. Of course, Pillay's theorem does note imply that every differential algebraic group can be given an affine embedding. The group operation of a definable group is not necessarily continuous in the Kolchin topology. 
\end{rem} 

Let $T$ be any theory and $\mc M$ a saturated model of $T$. If $\phi$ is a formula over $ \mc M$, then $B$ is a \emph{canonical base} for $\phi$ if $B$ is definably closed and whenever $\sigma \in Aut (\mc M )$ fixes $\phi (\mc M)$ as a set, $\sigma \in Aut(\mc M / B)$. A theory \emph{eliminates imaginaries} if every formula has a canonical base. The theory $DCF_{0,m}$ has elimination of imaginaries \cite[for discussion in the ordinary case, see Marker's article]{MMP}. In differential fields, \emph{canonical bases} correspond to minimal differential fields of definition. In fact, we may assume that every canonical base is a differential field, because definable closure of a given set in models of $DCF_{0,m}$ is obtained by taking differential field generated by that set. Though canonical bases are differential fields, for many applications, it is necessary to take a more detailed view and consider the generators of the differential field (for instance, this is necessary anytime Kolchin polynomials are to be considered). 

\begin{prop} (See \cite[page 53]{MMP}) \label{Elim} Suppose $T$ eliminates imaginaries and has two constant symbols. Let $E$ be a definable equivalence relation on $\mc M ^n$. Then there is $m \in \m N$ and a definable function $f: \mc M^n \rightarrow \mc M^m$ such that $E ( \bar x , \bar y)$ iff $ f( \bar x ) = f( \bar y )$. 
\end{prop}

Let us describe how to use the model theoretic tools we have described above to assign a Kolchin polynomial to an arbitrary differential algebraic group. We do not claim that the process described here is in any way canonical; the issues of this paper have to do with invariants of the Kolchin polynomial which are coarse enough to no be affected by the particular choices we make. Given a differential algebraic group $G$ over $k$, there is a finite $k$-open covering of $G$, $U_1, \ldots , U_d$ such that $U_i$ is isomorphic via a differential birational map $f_i$ (defined over $k$) to a Kolchin-closed subset of $\mc U^n$. The elements of the cover of $G$ are equipped with differential rational transition maps, so $G$ may be regarded as a definable quotient $ \sim$ of $\prod_{i=1}^d \mc U^n$ (recall that for each $i=1, \ldots ,d$, $U_i \subseteq \mc U^n$). Regarding $G$ as a differential algebraic variety $G=\prod_{i=1}^d \mc U^n / \sim$. Applying the previous proposition gives a map $G \rightarrow \mc U^m$ for some $m$. The Kolchin polynomial we assign to $G$ is a the Kolchin polynomial of the closure of the image of this map. There are various other manners in which one might attempt to assign a Kolchin polynomial to arbitrary differential algebraic groups; some methods might give a more canonical assignment, but the assignment described here suffices for the purposes of this paper. One might imagine that elimination of imaginaries also allows us to assign a Kolchin polynomial to $G/H$ when $H$ is a normal $k$-definable subgroup (or the left coset space of $H$ when $H$ is not normal) with no further ambiguities. However, the Kolchin polynomial calculations depend on the particular definable embedding given to a quotient. For the purposes of our paper, this ambiguity will not be an important issue, because we are dealing with invariants of the Kolchin polynomial which are differential birational invariants. 

\section{Stabilizers and fields of definition}
In this section, we develop the notion of stabilizers of types in the differential algebraic setting and give some new bounds on Kolchin polynomials. The basic setup is that of superstable groups \cite{Poizat}, but the proofs of the results are easier or sometimes give more information in the setting of differential algebraic groups. In this section, $G$ will be a differential algebraic group, that is, a definable group over some differential field, $k$. For discussion of the category of differential algebraic groups, see \cite{PillayDAG}. Again, recall that we assume $K$ is a small differentially closed field containing $k.$ For convenience of notation, when writing elements of definable groups, we will write $g \in G$ even though $g$ may be a tuple when thinking of $G$ as a definable (i.e., embedded) set. 

\begin{defn} Let $p(x) \in S (K)$ be a complete type containing the formula $x \in G.$ All of the complete types we deal with will contain this formula. Define 
$$stab_G (p) = \{ a \in G \, | \, \text{ if } b \models p, \, b \ind_K a, \text{ then } ab \models p \}.$$
\end{defn}

\begin{exam} Consider $G= Z(x'')$ in a model of the theory $DCF_{0,1}.$ $G$ is a subgroup of $\m G_a.$ In this case, we will write the group operation additively. Consider the generic type $p \in S_1(K)$ of $G.$ That is, $p \models x''=0,$ but not any lower order differential equations defined over $K.$  Then by definition, the stabilizer of $p$ in $G$ is the definable subgroup of elements $g \in G$ for which if $b \models p$ and $b \ind_K g$ then $g+b \models p.$ The independence of $g$ and $b$ ensures that $g+b$ satisfies no differential equations of order 1 (even over $K \langle g \rangle$). More generally, the stabilizer of the generic type of a connected $\omega$-stable group must be the entire group (for details, see chapter 2 of \cite{Poizat}). In this paper, we will be considering stabilizers of non-generic types of the differential algebraic group $G.$ In certain cases, we try to control these subgroups to get definability results.
\end{exam}

The next two results are standard for $\omega$-stable groups \cite{Poizat} (only the first needs $\omega$-stability; for the other, superstability suffices). 

\begin{lem} $stab_G(p)$ is definable. 
\end{lem}

\begin{lem} $RU(p) \geq RU(stab_G(p)).$ 
\end{lem}

The first task is to put the last lemma into the differential algebraic context, with differential type and typical differential dimension playing the role that Lascar rank plays in the model theoretic context. The next two lemmas are preparation for this result. 

Again, we are working in some fixed differential algebraic group $G$; \emph{all types, elements and tuples in the following lemma are assumed to be in the differential algebraic group in which we are working}. Unless specially noted, multiplication of two elements or above as in the case of a type occurs with \emph{respect to the group operation in} $G.$ 

\begin{lem}\label{lowerbound} Suppose that $c \ind _A b,$ where we assume that $A$ is a differential field which contains the canonical base of the differential algebraic group $G.$ Then $\tau(c /A) \leq \tau(cb/A)$ and in the case of equality, $a_\tau ({(c/A)}) \leq a_\tau ({(cb/A)}).$
\end{lem}
\begin{proof} $c$ is definable over $A   \langle b, cb \rangle$ and $cb$ is definable over $A \langle c,b \rangle.$ Definable closure is the same as differential field closure in our setting. So, $$\tau( c / A \langle b \rangle ) = \tau ( cb / A \langle b \rangle  ) \leq \tau(cb/A) $$ and in the case of equality in the previous line, $$a_\tau ({( c / A \langle b \rangle )})=a_\tau ({ ( cb / A \langle b \rangle )}) \leq  a_\tau ({ ( cb / A )}).$$ 

But, we know that $$\chi_{c/A \langle b \rangle}(t) =\chi_{c /A}(t)$$ by the characterization of forking in partial differential fields \cite[see also the discussion in the introduction of this paper]{McGrail}, so the lemma is established. 
\end{proof}

\begin{rem}
The fact that the Kolchin polynomial is not a differential rational invariant affects all $\Delta$-groups whose differential rational group law is not rational. So, for instance, left multiplication might not preserve the Kolchin polynomial of a group element. While it is true that differential algebraic groups may be embedded in algebraic groups \cite{PillayDAG}, relieving this potential problem, such an embedding (starting from our given differential algebraic group) need not preserve the Kolchin polynomial. One could state the above result with Kolchin polynomials, but only after assuming a specific embedding into an algebraic group.

Alternatively, a result of \cite{SitWell} shows that the Kolchin polynomials are well-ordered by eventual domination. For a given differential algebraic group, one could consider the set of Kolchin polynomials of groups isomorphic to $G.$ Selecting the minimal Kolchin polynomial from this set would give every differential algebraic group a canonical polynomial. Analysis of this polynomial from a model theoretic perspective seems might be possible since it is an invariant of a tuple up to interdefinability. We avoid both this approach and the one mentioned in the previous paragraph, because our results usually only require analysis of the differential type and typical differential dimension; when lower degree terms of the Kolchin polynomial are considered, more care is required (as the reader will see later in this section).
\end{rem}

The following lemma has been stated before \cite{PillayDAG} \cite{BenoistFrench} (though, to the author's knowledge, not precisely in this form). For instance, the proof in \cite{BenoistFrench} is for the case groups definable in ordinary differential fields. The proof in the partial differential version is not any harder, but we include it for convenience. 

\begin{lem}\label{Groups}  Suppose that G is an connected differential algebraic group. Then a type is generic in the sense of the Kolchin topology if and only if it is of maximal Lascar rank. 
\end{lem}
\begin{proof} We will refer to types which are generic in the Kolchin topology (in the sense that there is a realization of the type of Kolchin polynomial equal to the differential algebraic group) as a \emph{topological generic}. We will refer to the types of maximal Lascar rank as \emph{RU-generics}. We will refer to types for which any neighborhood covers the group $G$ by finitely many left translates as \emph{group generics}. In any superstable group, being group generic is equivalent to being RU-generic \cite{Poizat}.

Suppose that $p(x)$ is RU-generic but not topological generic.  Then finitely many left translates of any formula in $p(x)$ cover the group, but $p(x)$ is not topological generic, so the type is contained in a proper Kolchin closed subset of $G$. Take the formula witnessing this, $\phi(x)$. Now, finitely many left translates of $\phi(x)$ cover the group $G,$ and each of these is clearly closed in the Kolchin topology (if $a$ is a topological generic in $\phi(x)$ then $g\phi(x)$ is simply the zero set of the ideal of differential polynomials vanishing at $ag$). But, this is a problem. Now $G$ is the finite union of proper closed subsets.

Now, assume that $p(x)$ is a type such that any realization $a$ is topological generic. Then take any differential polynomial $P(x)$ vanishing at $a.$ As $a$ is topological generic, $P(x)$ vanishes everywhere in $G.$  So, by quantifier elimination, then only possible non-group generic formula in $p(x)$ is the negation of a differential polynomial equality. Suppose that $P(x)\neq 0$ is not group generic. Then $P(x)=0$ is group generic, so finitely many translates cover $G,$ which is again a contradiction if $Z(P) \cap G$ is a proper closed subset of $G$.  Thus $P(x)\neq 0$ is group generic. 
\end{proof}

From now on, we will simply refer to these types as generics. Note that this argument also shows that for a differential algebraic group, irreducibility in the Kolchin topology implies that the Morley degree is one \cite[for the definitions of Morley degree and rank]{MMP}. In a general differential algebraic variety (with no group structure), there are examples in which the topological generics are disjoint from the U-generics and Morley degree (and even Morley rank) of a definable (constructible) set is not preserved by taking the closure of the set in the Kolchin topology \cite{Fgenerics}. Also note that the previous lemma also holds for definable principal homogeneous spaces of a differential algebraic group \cite{PillayDAG}. 

\begin{prop}\label{smallstab} For any complete type which includes the formula ``$x \in G$", $\tau(stab_G(p(x))) \leq \tau(p(x))$ and in the case of equality, $a_\tau ( {stab_G(p(x))}) \leq a_\tau ({  p(x)}).$
\end{prop}
\begin{proof} Suppose that $s(x)$ is a generic type of the stabilizer of $p(x).$  Take $b \models p(x)$ and $c \models s(x)$ such that $b \ind _K c.$ $$\tau(c /K) \leq \tau(b c /K)$$ and if equality holds, then $$a_ \tau ({c /K}) \leq a_ \tau ({b c /K})$$  by Lemma \ref{lowerbound}. One needs only to argue that $tp( b c /K )=p(x).$ This follows from the definition of $stab_G(p(x)).$
\end{proof}

The following Lemma has not appeared in its precise stated form, but its proof has appeared \cite[see the proof of Proposition 1 on page 252]{SitPoly} \footnote{Thanks to the referee for pointing out this reference and the proof therein.}. This is the differential algebraic analogue (i.e., for the Kolchin polynomial) of the lower bound of the Lascar inequality. 

\begin{lem} \label{Sitlemma} (Sit's Lemma) Let $A$ be a differential field. Consider the following extensions of differential fields: 
$$
\xymatrix{
  A \langle  \bar a , \bar b \rangle \ar@{-}[d]  \\
  A \langle \bar b \rangle \ar@{-}[d] \\
  A
}
$$
Then $\chi _{\bar a / A \langle \bar b \rangle } (t) + \chi _ {\bar b /A } (t) \leq \chi _{(\bar a , \bar b) /A } (t).$
\end{lem}

\begin{rem} It is not even clear what analogue of the upper bound of the Lascar inequality \emph{should} be in this context. 
\end{rem} 

In the next theorem, we will discuss tuples which generate differential fields over which a given type does not fork (what might be called a relative canonical base or a relative field of definition). We remind the reader that Sit proved the Kolchin polynomials are ordered by eventual domination \cite{SitWell}. In what follows, we will write the Kolchin polynomial of a type $p$ in the following canonical form $$\chi_{p }(t) = \sum_{0 \leq i \leq m} a_i \binom{t+i}{i}.$$	
The following definition will be useful for the statements of the remaining results in the section. 

\begin{defn} Let $p \in S_N(k_1)$ and $q \in S_N(k_2)$ where $k_1 \leq k_2$ are differential fields and $q$ is an extension of $p$. Let $\chi_{p }(t) = \sum_{0 \leq i \leq m} a_i \binom{t+i}{i}$ and $\chi_{q }(t) = \sum_{0 \leq i \leq m} d_i \binom{t+i}{i}.$ We say that $p$ and $q$ are \emph{$n$-equivalent} if $a_i=b_i$ for all $i \geq n$. In this case, we also write $\chi_p (t) \equiv _n \chi _q (t)$. 
\end{defn}
 
\begin{rem} This new notion, $n$-equivalence, is a measure of forking; for instance, $0$-equivalence is equivalent to nonforking. $1$-equivalence of complete $n$-types means that the forking only changes the constant term of the Kolchin polynomial. The notion is only meaningful for $n \leq m$. 

We consider a simple example with $m=1$. Let $a ,b \in \mc U$ considered over the field $\m Q$. Suppose $\delta^2 (b) =0$ and $a$ is generic over $\m Q$. Now, consider the differential field $\m Q (c)$ where $\delta (b) = c$ (note that $\delta (c) =0$). Then $tp (a,b / \m Q (c) )$ is a forking extension of $tp (a,b / \m Q)$. The types are $1$-equivalent in this case, because all of the forking only affects the constant term of the Kolchin polynomial. The next theorem in part justifies the importance of this measure. 
\end{rem} 

In the proof of the following result, we will use the notion of a \emph{Morley sequence} in a type $p \in S_N(K)$, which we will briefly explain here \cite[for complete details]{Marker}. A sequence $\bar a_i \models p$ such that $\bar a_i \ind _K K \langle a_j \rangle _{j \neq i }$ is called a Morley sequence in the type $p$. 

\begin{thm}\label{first}  Suppose that $p(x) \in S_N(K)$.  Then, suppose, for some differential subfield $A \subseteq K$ and $n \in \m N$ that $\chi_{p|_A}(t) \equiv _n \chi_p(t)$. Then there is a tuple $\bar c \in K$ such that $\chi_p(t) =\chi_{p | _{ A \langle \bar c \rangle}}(t)$ and $\chi_{\bar c /A} (t)$ has degree less than $n$.
\end{thm}

\begin{proof} 
Let $\langle \bar b_k \rangle _{ k \in \m N }$ be a Morley sequence over $K$ in the type of $p.$ By the characterization of forking in $DCF_{0,m}$ this simply means that for all $j \in \m N$,  $$\chi_p(t)=\chi_{\bar b_j /K}(t)=\chi_{\bar b_j / K \langle \bar b_0, \ldots, \bar b_{j-1} \rangle}(t)$$
We do not know, however, that the same holds over the differential subfield $A \subseteq K.$ The sequence is still necessarily $A$-indiscernible, that is $tp(\bar b_j/A)$ does not depend on $j$. It is not necessarily $A$-independent (that is, it may be that $\bar b_i \nind _A \bar  b_j$), since $A$ may not contain a differential field of definition for the type $p$. In general, we simply know that $\chi_{\bar b_j / A  \langle \bar  b_0 , \ldots , \bar  b_{j-1} \rangle }$ is a decreasing sequence of polynomials, again, ordered by eventual domination. By the well-orderedness of Kolchin polynomials \cite{SitWell} we know that the sequence is eventually constant. Alternatively, this fact can be seen by noting the superstability of $DCF_{0,m}$ and the fact that decreases in Kolchin polynomial correspond to forking extensions. So, for the rest of the proof, we fix a $k$ such that if $n \geq k,$ the sequence is constant. That is, above $k,$ we know that we have a Morley sequence over $A  \langle b_0, \ldots ,b_{k-1} \rangle$ in the type of $p.$ Now, fix a model $K' \models DCF_{0,m}$ with $K'$ containing $K$ and $\{b_0 , \ldots , b_{k-1} \}$. We let $p'$ be the (unique) nonforking extension of $p$ to $K'.$ 

We can get elements $\bar c \subseteq acl^{eq}(A \langle b_0, \ldots , b_{k-1} \rangle)$ such that $p'$ does not fork over $\bar c.$ In fact, by \cite{Shelah} (page 132) and the fact that $DCF_{0,m}$ eliminates imaginaries, we can assume that $\bar c \in K.$ In differentially closed fields, the algebraic closure of a set $E$ over $A$ is equal to the field theoretic algebraic closure of the differential field generated by $E$ over $A$. If $\bar c \in acl^{fields} (A \langle E \rangle )$, it follows that $ \bar c \in acl^{fields} (A (\Theta (r) (E) ) )$ for some $r \in \m N$ (note that by $ \Theta (r) E,$ we mean $\{\theta (e) \, | \, \theta \in \Theta (r), \, e \in E \}$). Take such a minimal $r$. Now, we will replace each $b_i$ by $ \Theta (r) b_i$ for an $r$ satisfying the above requirement (order $\Theta (r) b_i$ as a tuple via the standard ordering on differential monomials if you like). Note that this has no effect indiscernibility (though the length of the indiscernible tuples has changed) nor on the value of $k$ from the above paragraph. This tuple $\Theta (r) b_i$ is interdefinable with $b_i$ and forking is invariant up to interalgebraicity. Suppose that $b_0 \models p$; then replacing $p$ with $\Theta (r) p :=tp ( \Theta(r) b_0/K)$ preserves the assumption of $n$-equivalence because the Kolchin polynomial of $\Theta(r) p$ is $\chi_p(t+r)$ and the Kolchin polynomial of $\Theta (r) p |_A$ is equal to $\chi _{p|_A} (t+r)$. For convenience, we will now drop the notation $\Theta (r) p$, but we are now assuming that $ \bar c \in acl^{fields} (b_0 , \ldots b_{k-1} /A)$. This point will be rather important in what follows. 

By assumption, $\chi_{p|_A}(t) \equiv _n \chi_p(t),$ so the polynomial $h_1(t): =\chi_{p|_A}(t) - \chi_p(t)$ is of degree less than $n$. Note that since we may assume $\chi_{p|_A}(t)$ is greater than $\chi_p(t)$, the leading coefficient is positive; this follows by noting the result of \cite{SitWell} which says that when written as a sum of binomials, the eventual domination ordering of Kolchin polynomials is the same as lexicographic ordering of the tuples of coefficients. 

By construction $\langle b_i \rangle$ was an indiscernible sequence, so if we define $\bar b := (b_0, \ldots ,b_{k-1} ),$ then because $ \bar c$ the canonical base for $p$ over $A$ and $\bar b$ comes from an initial segment of an independent indiscernible sequence in $p$, \begin{eqnarray} \label{ineq} k \cdot \chi_p (t) = \chi_{\bar b / A \langle \bar c \rangle }(t) = \chi_{\bar b /K}(t) \end{eqnarray}

By assumption, $\bar c \in acl^{fields}(A  (\bar  b ) )$. Since algebraic elements do not affect Kolchin polynomials, $\chi_{\bar b/A}(t) =\chi_{\bar b , \bar c /A}(t).$  

Then by Sit's Lemma \ref{Sitlemma}, \begin{eqnarray}\label{finalineq} \chi_{\bar b/A \langle \bar c \rangle }(t) + \chi_{\bar c/ A}(t) \leq \chi_{\bar c, \bar b/A}(t)= \chi_{\bar b/A}(t).\end{eqnarray} 

We claim that 
\begin{eqnarray}\label{laterineq} \chi_{\bar b/A}(t) \leq k \cdot \chi_{p|_A}(t). \end{eqnarray} This follows from noting that if $ \bar b$ was an independent set of realizations of $\chi_{p|_A}(t)$, then the Kolchin polynomial of $\bar b$ over $A$ would be equal to $k \cdot \chi_{p|_A}(t)$. If the $b_i \in \bar b$ are not an independent over $A$, then the Kolchin polynomial for $ \bar b$ can only be lower. So, by \ref{ineq}, \ref{finalineq} and \ref{laterineq}, we see
\begin{eqnarray*} k \cdot \chi_p (t) + \chi_{\bar c/ A}(t) & \leq & k \cdot \chi_{p|_A}(t) \\
 \chi_{\bar c/ A}(t) & \leq & k \cdot \chi_{p|_A}(t) -k \cdot \chi_p (t) = k \cdot h_1 (t)
\end{eqnarray*} 
But, the polynomial $h_1$ is of degree less than $n$. 

\end{proof}

\begin{rem}
Let us briefly discuss the significance of the previous result. The theorem says that any forking which only affects the terms of the Kolchin polynomial of degree below $n$ can be achieved by adding a tuple $\bar c$ to the canonical base of the restricted type which itself has Kolchin polynomial of degree less than $n$. 

Let $a \models p$. Then we have the following diagram:

$$
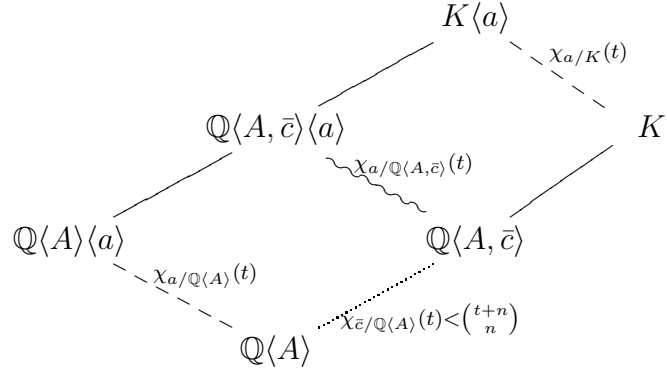
\xymatrix{
& & K \langle a \rangle \ar@{--}[rd]^{\chi_{a/K}(t)} \ar@{-}[ld]  & \\
&  \m Q \langle A, \bar c \rangle \langle a \rangle \ar@{~}[rd]^{\chi_{a/\m Q \langle A, \bar c \rangle}(t)} \ar@{-}[ld] & & \, \, \, \, \, \, K \ar@{-}[ld] \\
 \m Q \langle A \rangle \langle a \rangle \ar@{--}[rd]^{\chi _{a /\m Q \langle A \rangle} (t) }  &  & \m Q \langle A, \bar c \rangle \ar@{.}[ld]^{\chi _{\bar c /\m Q \langle A \rangle} (t) < \binom{t+n}{n}} &\\
&  \m Q \langle A \rangle & &
}
$$
\end{rem}

The next lemma appears in \cite{BerlineLascar}:

\begin{lem}\label{second} Let $p(x) \in S(K)$ with $``x \in G" \in p(x).$  Suppose $p$ does not fork over the empty set. Let $b$ be an element of $G (K).$ Let $A \subset K$. If $\tilde b = b \, (mod \, stab_G(p))$ is not algebraic over $A,$  then $bp$ forks over $A.$ 
\end{lem}

\begin{prop}\label{KeyLemma} Let $G$ be a differential algebraic group and suppose that $a,b \in G$ are such that $\tau ( a/K)= \tau (b /K) = n $ and $ a \ind _K b$. Suppose also that $\tau (ba /K ) = n $ and $a_n (ba /K) = a_n (a/K)$. Let $p = tp (a/K)$ Let $ \tilde b \in G/stab_G(p)$ be the left coset of $b$ modulo $stag_G (p)$. Then $\tau ( \tilde b /K) <n$. 
\end{prop} 
\begin{proof} Let $K'$ be a differentially closed extension field of $K$ in $\mc U$ which contains $K \langle b \rangle $ and is linearly disjoint from $K \langle a \rangle $ over $K$. Then $tp (a / K')$ does not fork over $K$. That is $chi_ {a/K} = \chi _{a/K'}$. So, $\tau (a/K'=n$ and $a_\tau (a/K')= a_\tau (a/K).$ We can also see that $K ' \langle a \rangle = K' \langle ba \rangle, $ so 
$\tau (ba /K') = \tau (a/K')$ and $a_\tau (ba/K') = a_\tau (a/K')$. By hypothesis, $\tau ( ba/K') = \tau (ba /K) =n$ and $a_\tau (ba /K' ) = a_\tau (ba / K).$ Now, $\chi _ {ba /K' } \equiv _n \chi _{ba/K}.$ 

By theorem \ref{first} we can get $\bar c \in K'$ with $\chi_{\bar c /K}(t)$ of degree less than $n$ such that $ba \ind _{K \langle \bar c \rangle } K'.$ Then, applying lemma \ref{second} we can see that $\tilde b$ is algebraic over $K \langle \bar c \rangle.$ We know that  $\chi_{\bar c /K}(t)$ is of degree less than $n$ and so $\chi_{\tilde b /K}(t)$ is as well. 

\end{proof}

\section{Indecomposability}
\begin{defn}\label{IndecomposableDefn}
Let $G$ be a differential algebraic group defined over $K.$  Let $X$ be a definable subset of $G$. For any $n \in \m N,$ $X$ is $n$-indecomposable if $\tau(X/H) \geq n $ or $|X/H|=1$ for any definable subgroup $H \leq G.$  We use \emph{indecomposable} to mean $\tau(G)$-indecomposable.
\end{defn} 

Note that $n$-indecomposability implies irreducibility. 

Elimination of imaginaries can be used to justify the notation $\tau (X/H).$ For any two elements $x_1$ and $x_2$ of $X,$ say $x_1 \sim x_2$ if $x_1$ and $x_2$ are in the same left coset of $H.$ This is a definable equivalence relation, so the set $X/H$ has differential algebraic structure. So, one may apply \ref{combinatorialprep} to assign $\tau (X/H).$ For more regarding elimination of imaginaries in differentially closed fields, see \cite{MMP}. Kolchin \cite{KolchinDAG} also constructs quotient objects, and his approach could be used here as well. 

We should note the relationship of this notion to that of \emph{strongly connected} considered by \cite{CassidySinger}. A subgroup of $G$ is strongly connected if and only if it is $n$-indecomposable where $n =\tau(G).$ We will occasionally use \emph{$n$-connected} to mean $n$-indecomposable, but only in the case that the definable set being considered is actually a \emph{subgroup}. In the next section, we will show some techniques for constructing indecomposable sets from indecomposable groups. 

\begin{defn} A differential algebraic group $G$ is \emph{almost simple} if for all proper normal differential algebraic subgroups $H$ of $G,$ we have that $\tau(H) < \tau (G).$ 
\end{defn} 

We note that by \ref{quot}, almost simple is a strengthening of strong connectedness. In the definition of strong connectedness, one could take the subgroup $H$ to be normal without changing the definition. The reason for this is that, given a differential algebraic group $G$, then the set of subgroups such that $\tau (G/H) < l$ is closed under finite intersection. From this observation and the basis theorem for the Kolchin topology (or the Baldwin-Saxl condition), one can prove the existence of the strongly connected component of the identity for an arbitrary differential algebraic group. All of the preceding discussion in this section comes from \cite{CassidySinger}. 

In the definition of almost simple, taking the subgroup $H$ to be normal is necessary. For example, every differential algebraic group $G$ has an commutative differential algebraic subgroup of the same $\Delta$-type \cite{JIndecomp2}.

Cassidy and Singer proved the following, showing the robustness of the notion of strong connectedness under quotients,
\begin{prop} Every quotient $X/H$ of a $n$-connected definable normal subgroup $X$ by a definable subgroup is $n$-connected.
\end{prop}
For many other properties of strongly connected subgroups and numerous examples, see \cite{CassidySinger}. The next proposition is used in the proof of the main theorem, but is stated separately because it applies more generally. 

\begin{prop}\label{general} Let $X_i$, for $i \in I$, be a family of $l$-indecomposable sets of a differential algebraic group $G$ for some $l \in \m N.$ Assume each $X_i$ contains the identity. Let $H =\langle X_i \rangle$ and \emph{suppose} that $H$ is definable. Then $H$ is $l$-indecomposable.
\end{prop}
\begin{proof} Let $H_1 \leq G$ with $H \not \leq H_1.$ Then there exists $i$ such that $X_i \not \subset H_1.$ For this particular $i,$ we know, by $l$-indecomposability, that the coset space $X_i/H_1$ has differential type at least $l$. That is $\tau (X_i /H_1) \geq l.$ But, then $\tau (H /H_1) \geq l$ since $H$ contains $X_i.$ 
\end{proof}

Note that there is no assumption in \ref{general} that $\tau (G) =l.$ This is assumed in the indecomposability theorem \ref{indecomposability}, but the proposition about the $l$-indecomposability of the generated subgroup holds more generally, assuming the group is definable. In general, we do not know about the definability of such a subgroup, unless additional assumptions are made.

\begin{thm}\label{indecomposability} Let $G$ be a differential algebraic group. Let $X_i$ for $i \in I$ be a family of $n$-indecomposable definable subsets of $G$ with $\tau(G)=n$. Assume that $1_G \in X_i.$ Then the $X_i 's$ generate a strongly connected differential algebraic subgroup of $G.$  
\end{thm}
\begin{proof} Since we are considering the the group generated by the collection of sets $X_i$, without loss of generality, assume that for each $i$, there is a $j$ such that $X_i^{-1} = X_j$. 
Fix $n=\tau(G).$ Let $\Sigma$ be the set of finite sequences of elements of $I,$ possibly with repetition. Then for $\sigma \in \Sigma$ with $length(\sigma)=n_1,$ we let $X_\sigma=X_{\sigma (1)} \cdot \ldots \cdot X_{\sigma (n_1)}.$ 
Let $k_\sigma =a_\tau({ X_ \sigma}).$ We note that $k_\sigma \leq a_\tau(G).$ For the remainder of the proof, we let $\sigma \in \Sigma$ be such that $$k_{\sigma }=Sup_{\sigma \in \Sigma} (k_\sigma)$$ is achieved; this supremum is simply the maximum of a finite set of nonnegative integers since the differential types of the finite products are all equal to $\tau (G)$. Note that $\tau (X_ \sigma =n$ and $a_\tau (X_ \sigma) = k_ \sigma$ and $X_ \sigma $ is also an $n$-indecomposable subset. 

Now, let $p \in X_{\sigma}$ such that $\tau(p) =n$ and $a_p=k_{\sigma}$ (here we are regarding $p$ as a point and $X_ \sigma$ as a subset of the Stone space). 
We consider $stab_G(p).$ First, we note that $$stab_G(p) \subseteq X_{\sigma} X_{\sigma}^{-1}$$ To see this, let $b \in stab_G(p)$ and $a \models p.$ Then both $a$ and $ba$ satisfy $X_{\sigma}.$ Then $b=baa^{-1} \in X_{\sigma} X_{\sigma}^{-1}.$ Next, we will show, for all $i,$ 
$$X_i \subset stab_G(p) \text{ for all $i \in I$}$$
Let $b \in X_i$ be generic over $K$. We may suppose without loss of generality that $a \ind _K b$. Then $ba$ is generic for $X_i X_ \sigma$ over $K$. Therefore $\tau(ba) = \tau (a) = n.$ As we remarked above $X_ \sigma \subseteq X_i X_ \sigma $. So, the maximality of the choice of $a_\tau ( X_ \sigma)$ gives that $a_ \tau (ba /K) = a_ \tau ( X_i X_\sigma ) = k_ \sigma.$

We claim that $\tau (\tilde b ) < \tau (a) =n$ where $\tilde b$ is the class of $b$ modulo $stab_G(p).$ This follows by applying Proposition \ref{KeyLemma} and noting that $p$ has the properties needed for the hypothesis of that proposition, by the maximality of the differential type of $p \in X_{\sigma}$. The space of types satisfying $X_i$ maps surjectively onto $X_i /stab_G(p)$. The map also preserves genericity since each definable subset of $X_i /stab_G(p)$ is the image of a definable subset of $X_i$, so if $\tilde b$ were not generic, then $\tilde b \in Y_i/stab_G (p) $ for some proper $K$-definable $Y_i \subseteq X_i,$ and so $b \in Y_i$, contradicting the genericity of $b$. Now we have a a contradiction to the $n$-indecomposability of $X_i$ unless $X_i \subseteq stab_G (p)$. 
\end{proof}

Of course, as in the more familiar case of groups of finite Morley rank we know more than that the group generated by the family is definable. We have constructed the definition of the group which gives it a very particular form. For further discussion see \cite{Marker}, chapter 7, section 3. Analogies between strongly connected differential algebraic groups and groups of finite Morley rank will be pursued in \cite{JIndecomp2}.

\begin{rem} The notion of $n$-indecomposable presented here is similar (and inspired by) the notion of $\alpha$-indecomposable considered by \cite{BerlineLascar}. It might be the case that this notion is a specialization of the notion considered there (note that for differential fields, the Lascar rank of a type is always less than $\omega^m$, where $m$ is the number of derivations). Proving that the notions coincide for differential fields would require finding a lower bound for Lascar rank in terms of $\Delta$-type. There is currently no known lower bound for general differential varieties. For a discussion of this issue see \cite{DeltaCompleteness}. Even if such a lower bound is found, there are compelling reasons to develop indecomposability in this manner. In difference-differential fields, it is known that no such lower bound for Lascar rank holds, and so the notions of indecomposability coming from the analogue of the Kolchin polynomial and the notion coming from Lascar rank will be distinct. 
\end{rem}

\section{Definability of derived subgroups of strongly connected groups}
In this section, we will first show some applications of the ideas and techniques for constructing indecomposable sets. Any group naturally acts on itself by conjugation, that is $x \mapsto gxg^{-1}.$ For a subset $X \subseteq G$, we define 
$$X^ G : = \{ g x g^{-1} \, | \, x \in X , \, g \in G \}.$$ 
Analysis of this action provides a way of transferring properties of the group doing the action to the set on which it acts. Now, fix a differential algebraic group $G$ and a differential algebraic subgroup $H.$ A subset $X \subseteq G$ is \emph{$H$-invariant} if for all $h \in H,$ conjugation by $h$ is a bijection from $X$ to itself, that is $H$ is contained in the normalizer of the set $X$. 

First, we give the following example, due to Cassidy \cite[page 110]{KolchinDAG}, of a differential algebraic group for which the derived subgroup is not a differential algebraic group. 
\begin{exam}
Let $\Delta=\{\delta_1, \delta_2\}.$ Then consider the following group $G$ of matrices of the form: 
\[ \left( \begin{array}{ccc}
1 & u_1 & u \\
0 & 1 & u_2 \\
0 & 0 & 1 \end{array} \right)\] where $\delta_i (u_i)=0.$ 
Of course, 
\begin{eqnarray*}
& \left( \begin{array}{ccc}
1 & u_1 & u \\
0 & 1 & u_2 \\
0 & 0 & 1 \end{array} \right)
\left( \begin{array}{ccc}
1 & v_1 & v \\
0 & 1 & v_2 \\
0 & 0 & 1 \end{array} \right) 
\left( \begin{array}{ccc}
1 & u_1 & u \\
0 & 1 & u_2 \\
0 & 0 & 1 \end{array} \right)^{-1}
\left( \begin{array}{ccc}
1 & v_1 & v \\
0 & 1 & v_2 \\
0 & 0 & 1 \end{array} \right)^{-1} \\
& =\left( \begin{array}{ccc}
1 & 0 & u_1 v_2 -v_1 u_2 \\
0 & 1 & 0 \\
0 & 0 & 1 \end{array} \right)
\end{eqnarray*}
Then one can see that the derived subgroup is isomorphic $\m Q (C_{\delta_1} \cup C_{\delta_2}),$ where $C_{\delta_i}$ is the field of $\delta_i$-constants. This is not a differential algebraic group. This group is not strongly connected, however, since the subgroup of matrices of the form:
\[ \left( \begin{array}{ccc}
1 & 0 & u \\
0 & 1 & 0 \\
0 & 0 & 1 \end{array} \right)\]
is a subgroup of $\Delta$-type and typical $\Delta$-dimension equal to $G.$ This means that the coset space has $\Delta$-type strictly smaller than $G.$ Of course, this means that $G$ is not almost simple (or even strongly connected). Theorem \ref{Dcomm} will show that this sort of example is impossible for strongly connected differential algebraic groups.
\end{exam}

The next two lemmas have similar proofs in the context of groups of finite Morley rank (see Chapter 7 of \cite{Marker}).

\begin{lem}\label{Invar} Let $\tau (G) =n$. Let $X$ be $H$-invariant. Suppose for all $H$-invariant differential algebraic subgroups $H_1 \leq G,$ that $|X/H_1|=1$ or $\tau (X/H_1)\geq n.$ Then $X$ is $n$-indecomposable.
\end{lem}
\begin{proof} Suppose that there is a differential algebraic subgroup $H_2 \leq G$ with $\tau(X/H_2) <n,$ but $|X/H_2 | \neq 1.$ Then, by the $H$-invariance of $X,$ if $h \in H$ then $x^h \in X.$ Thus $h$ defines a map from $X/H_2 \rightarrow X/H_2^h.$ In particular, $\tau(X/H_2^h) <n,$ but $|X/H_2^h| \neq 1.$ Then, set $$H_1 = \bigcap_{h \in H} H_2^h$$ Then, by the Baldwin-Saxl condition (or by the Notherianity of the Kolchin topology), we know that $H_1$ is actually definable and is, in fact, the intersection of finitely many of the subgroups. But then, $H_1$ is clearly $H$-invariant and there is a differential birational morphism: 
$$X/H_2  \rightarrow  X / (h H_2 h^{-1} )$$
given by (noting the $H$-invariance of $X$),
$$x H_2 \mapsto h (x H_2 )h^{-1} = hxh^{-1} h H_2 h^ {-1} = x_1 h H_2 h^{-1} .$$

It follows that $\tau (X/ h H_2 h^{-1} ) = \tau ( X/H_2) <n$. Write the intersection, $H_1$ in the form $ \cap _{j=1}^ r (h_j H_2 h_j ^{-1})$. We have a differential rational embedding from 
$$X/H_1 \rightarrow \Pi _{j=1}^ r X/ (h_j H_2 h_j ^{-1})$$ 
given by $$x H_1 \mapsto  \Pi _{j=1}^ r x (h_j H_2 h_j ^{-1}).$$
It then follows immediately that $\tau (X/H_1)<n$ and $|X/H_1| \neq 1 ,$ contradicting the assumptions on $X.$   
\end{proof}

For $g \in G$ and $H$ a subgroup of $G$, let $g^H$ be the orbit of $g$ under $H$ via the action of conjugation, 

$$g^ H : = \{ h g h^{-1} \, | \,  h \in H \}.$$ 

\begin{lem} \label{4.3} Let $G$ be a differential algebraic group of differential type $n$, and let $g \in G$. Let $H$ be an $n$-indecomposable differential algebraic subgroup of $G$. Then the orbit $g^H$ of $g$ under $H$ is $n$-indecomposable. 
\end{lem}
\begin{proof}
The set $g^H$ is $H$-invariant. Using the previous result, it is enough to prove the result for all $N \leq G$ which are $H$-invariant. So, to that end, suppose that $N$ is such that $|g^H /N|\neq 1$ and $\tau(g^H/N) <n.$ Now, we get, by the $H$-invariance of $g^H$ and $N,$ a transitive action of $H$ on $g^H /N$, $$h*g^{h_1} N \mapsto h g^h_1 N h^{-1} = h g^h h^{-1} h N h^{-1} =g^{ h h_1} N$$ Thus, this is a transitive action of $H$ on a differential algebraic variety of differential type less than n. The centralizer of $g$ in $H$ must be a subgroup of $H$ of differential type $\tau(H)$ and typical differential dimension equal to that of $H.$  This is impossible, by the indecomposability of $H$, unless the kernel is simply $H$ itself (see \cite{CassidySinger} or \cite{JIndecomp2}). If that is the case, then by the transitivity of the action, $|g^H/N|=1.$
\end{proof}

Cassidy and Singer make the following comment in \cite{CassidySinger}, ``We also do not have an example of a non-commutative almost simple linear differential algebraic group whose commutator subgroup is not
closed in the Kolchin topology." The next result shows that such an example is not possible, even in the more general case of the group being strongly connected (with no assumption of linearity or almost simplicity). 

\begin{thm}\label{Dcomm}
Let $G$ be a strongly connected non-commutative differential algebraic group. The derived group of $G$ is a strongly connected differential algebraic subgroup.
\end{thm}
\begin{proof}
Let $\tau (G) = n$. $G$ is strongly connected, so by Lemma \ref{4.3}, for $g \in G$, $g^G$ is $n$-indecomposable. So, $X_g:= g^G g^{-1}$ is $n$-indecomposable and contains the identity. By Theorem \ref{indecomposability}, the group generated by the family $(X_g)_{g \in G}$ is a strongly connected differential algebraic subgroup of $G$. 
\end{proof}

We should also note the result of Cassidy and Singer which says that if a strongly connected differential algebraic group is not commutative, then the differential type of the differential closure of derived subgroup is equal to the differential type of the whole group \cite{CassidySinger}. So, putting this together with the above theorem yields: 

\begin{corr} Let $G$ be a strongly connected noncommutative differential algebraic group. Then the derived subgroup of $G$ is a strongly connected differential algebraic subgroup with $\tau([G,G])=\tau(G).$ 
\end{corr}
Because derived subgroups are characteristic (thus normal), they are candidates to appear in the Cassidy-Singer decomposition of $G$ (see \cite{CassidySinger}. We also get the following generalization of a theorem of Cassidy and Singer (who proved it in the case of an almost simple linear differential algebraic groups of differential type at most one). 

\begin{thm}
Let $G$ be an almost simple differential algebraic group. Then $G$ is either commutative or perfect. 
\end{thm}
\begin{proof} $\tau([G,G])= \tau(G)$ implies that $[G,G]=G,$ since $G$ is almost simple. Otherwise $[G,G]=1.$  
\end{proof}

Explicit calculations of the Kolchin polynomial for differential algebraic subgroups of $\m G_a ( \mc U)$ are often easier than for general differential algebraic groups or varieties. We will briefly describe how to perform these calculations, and how they lead to many examples of indecomposable (strongly connected) differential algebraic groups. There is a well-developed machinery for doing such calculations, for instance see \cite{kondratieva1998differential}. 

The machinery is particularly easy to deal with in the case that $G$ is the zero set of a single linear homogeneous differential polynomial in a single differential indeterminate, that is, $G$ is given as the zero set of $f(z) \in K \{z\}$. Suppose that, for some orderly ranking of the differential polynomial algebra $K \{ z \}$, the leader of $f(z)$ is $\delta_1^{i_1} \ldots \delta_m^{i_m} z.$ Let $r= \sum_{j=1}^n i_j $.

\begin{prop} \cite[Proposition 4 page 160]{KolchinDAAG} With $f$ given as above, let $V$ be the zero set of $f$. Then $$\chi_V (t) = \binom{t+m}{m} - \binom{t - r+m}{m}.$$
\end{prop} 

So, $\tau (G) = m-1$ and $ a_ \tau (G) =r$. Now let $H$ be a proper definable subgroup. Since the containment is proper, we must have that $\chi_H < \chi _G .$ Now, suppose that $\tau (H) = \tau (G) = m-1$ and $a_ \tau ( H) = a_ \tau ( G) = r$. Take $a \in H$ generic over a field $K$ containing the field of definition of $G$ and $H$. Then $\tau ( a/K)= m-1$ and $a_ \tau (a/K) = r$. The following is the theorem of Mikhalev and Pankratiev:

\begin{thm} \cite{MP1980} The minimal Kolchin polynomial in the set of all Kolchin polynomials of differentially finitely generated differential field extensions of $\mc F /K$ of degree $m-1$ (again, $m$ is the number of derivations) and typical differential dimension $r$ is $\binom{t+m}{m} - \binom{t - r+m}{m}$. 
\end{thm} 

This contradicts our assumption that $\tau (H) = \tau (G) = m-1$ and $a_ \tau ( H) = a_ \tau ( G) = r$, and establishes: 

\begin{prop}\label{Suersub} Let $G$ be a differential algebraic subgroup of $\m G_a ( \mc U )$ defined by a single differential polynomial. Then $G$ is strongly connected. 
\end{prop} 

\begin{exam}
We will again work over a model of $DCF_{0,2}.$ The following example was explored in \cite{CassidySinger} and was originally given in \cite{Blumberg}. Let $G$ be the solution set of $$(c_2 \delta_1^3 -c_2 \delta_1^2 \delta_2 - 2 c_2 \delta_1 \delta_2 +c_2^2 \delta_2^2 +2 \delta_2)x =0 $$ where $\delta_1 c_2=1$ and $\delta_2 c_2 =0.$ By the above discussion of linear homogeneous differential equations, this is a strongly connected differential algebraic group. There are differential algebraic subgroups of $\Delta$-type 1. In fact, since  
\begin{eqnarray*}
& &  c_2 \delta_1^3 -c_2 \delta_1^2 \delta_2 - 2 c_2 \delta_1 \delta_2 +c_2^2 \delta_2^2 +2 \delta_2\\
& &= (c_2 \delta_1 -c_2^2 \delta_2 -2)(\delta_1^2 -\delta_2)\\
& &=(c_2 \delta_1^2 -c_2 \delta_2 -2 \delta_1)(\delta_1 -c_2 \delta_2),\\
\end{eqnarray*}
the solution sets to $\delta_1^2 x -\delta_2 x=0$ and $\delta_1 x -c_2 \delta_2 x =0$ are differential algebraic subgroups. Note that $\delta_1^2 x -\delta_2 x=0$ is the heat equation. It follows from Proposition \ref{Suersub} that this subgroup is strongly connected. In \cite{Suer2007Article} Suer showed the solution set of the heat equation has Lascar rank $\omega$ by showing that every definable subset has finite transcendence degree. This actually establishes the stronger fact that the subgroup is almost simple. 

Our Proposition \ref{Suersub} also shows that the group $G$ defined by $\delta_1 x -c_2 \delta_2 x =0$ is strongly connected. But, in this case, $(\tau(G) , a_\tau (G) ) = (1,1)$ strong connectivity implies that any proper definable subgroup $H$ has $(\tau (H) , a _\tau (H) ) <_{Lex} (\tau (G) , a_ \tau (G))$. Then $\tau (H) = 0$ or $\tau (H) =-1$. So, $G$ is almost simple. One can also see this fact more directly in this case. We will show that the subgroup given by the solutions to $\delta_1 x -c_2 \delta_2 x =0$ only has finite transcendence degree definable subsets (and is thus almost simple). This subgroup is irreducible in the Kolchin topology, so the only definable proper subsets correspond to forking extensions of the generic type of subgroup. But, modulo, $\delta_1 x -c_2 \delta_2 x =0$, any differential polynomial can be expressed as a $\delta_2$-polynomial or a $\delta_1 $-polynomial.
\end{exam}

\section{Another Definability Result}
In this section, we prove results inspired by work of Baudisch \cite{Baudisch}. As with many of the results of this paper, the relationship between the results here and the existing work on superstable and $\omega$-stable groups would only become clear by getting control (or showing counterexamples) of Lascar rank in terms of differential type. 

Given a differential algebraic group $G$ with $\tau (G) =n$, there is a minimal definable subgroup $G_n \leq G$ such that $\tau (G/G_n) < \tau (G)$. This minimal subgroup $G_n$ is called the \emph{strongly connected component} of $G$.The strongly connected component of a differential algebraic group $G$ is a characteristic (and thus normal) subgroup \cite{CassidySinger}. 

\begin{thm} Suppose $\tau(G)=n$ and $G_n$ is the strongly connected component of $G.$ Suppose that $H$ is an abstract group such that $H \lhd G_n$. Then if $H$ is simple and non-commutative, $H$ is definable. 
\end{thm}
\begin{proof} Let $h \neq 1$, $h \in H.$ Then we will show $h^{G_n}$ is indecomposable. By Lemma 4.2 we only need to consider normal subgroups $N \leq G_n$. 
So, let $N$ be a definable normal subgroup of $G_n.$ 
First, suppose that $N \cap H \neq 1.$ Then because $H$ is simple and non-commutative, then, $N \cap H = H$ or $ N \cap H = 1_ G$. 
In the first case, the coset space $| h^{G_n} \cup \{ 1\} /N|=1.$ Thus, we may assume that $N \cap H =1$, so $h ^ {G_n}/N$ is differentially rationally isomorphic to $h^{G_n}$. Now, to verify that $h^{G_n}$ is indecomposable, we only need to show that $\tau(h^{G_n}) =n.$

There is a bijection between the elements of $h^{G_n}$ and the $G_n$-cosets of $C_{G_n} (h).$ So, it would suffice to prove that $\tau (G_n/C_{G_n}(h))=n.$ 
We know that $H \not \leq C_{G_n} (h),$ because $H$ is simple. 
But, then $|G_n /C_{G_n}(h)| \neq 1.$ 
Because $G_n$ is strongly connected, and $C_{G_n}(h)$ is a definable subgroup, $\tau(G_n/ C_{G_n}(h))=n.$ 
Now, we know that the following family of definable sets $\langle h^{G_n} \rangle _h \in H$ is indecomposable. Now we apply Theorem \ref{indecomposability} to see that $H$ must be definable.
\end{proof}
Further definability consequences of indecomposability will be pursued in \cite{JIndecomp2}. 

\section{Generalizations of strong connectivity} 
For differential algebraic groups, the notion of indecomposable matches the notion of strongly connected. But, in Definition \ref{IndecomposableDefn} we defined $n$-\emph{indecomposable} for arbitrary $n$, not necessarily equal to the differential type of the ambient differential algebraic group. In this section we will explore the notion in the case that $n \neq \tau(G).$ Consider the following family of proper differential algebraic subgroups, 

$$\mc G_n:= \{ H < G \, | \, \tau(G/H) < n \}$$ 

We note that this family is closed under finite intersections \cite{CassidySinger}. Since $G$ is an $\omega$-stable group, 
$$\bigcap _{H \in \mc G_n } H $$
is a definable subgroup, which we will denote $G_n.$ We note that $G_n$ is a characteristic subgroup of $G$. We will refer to $G_n$ as the $n$-connected component. When $n=\tau (G)$, the $n$-connected component is the strongly connected component, used earlier in the paper.

\begin{exam} 
For simplicity, in this example, suppose that $\Delta = \{ \delta _1 , \delta _2 , \delta _3 \}$. 
It is entirely possibly that the subgroups $H_n$ are different for every $n.$ The following is a very simple example which readily generalizes. Consider the set of matrices of the form 
\[ \left( \begin{array}{cccc}
1 & u_{12} & u_1 & u\\
0 & 1 & 0  & u_2\\
0 & 0 & 1 & u_{23}\\
0 & 0 & 0 & 1 \end{array} \right)\]
where $\delta_1 \delta_2 u_{12}=0,$ $\delta_1 u_1=0$, $\delta_2 u_2 =0,$ and $\delta_2 \delta_3 u_{23}=0.$

This set of matrices is closed under matrix multiplication since \begin{eqnarray*} \left( \begin{array}{cccc}
1 & u_{12} & u_1 & u\\
0 & 1 & 0  & u_2\\
0 & 0 & 1 & u_{23}\\
0 & 0 & 0 & 1 \end{array} \right) \cdot \left( \begin{array}{cccc}
1 & h_{12} & h_1 & h\\
0 & 1 & 0  & h_2\\
0 & 0 & 1 & h_{23}\\
0 & 0 & 0 & 1 \end{array} \right) \\ = 
\left( \begin{array}{cccc}
1 & h_{12}+u_{12} & u_1 +h_1 & h+u_{12}h_2 +u_1 h_{23} +u\\
0 & 1 & 0  & u_2+h_2\\
0 & 0 & 1 & u_{23}+h_{23}\\
0 & 0 & 0 & 1 \end{array} \right) \end{eqnarray*}
and the coordinates evidently satisfy the same differential equations as the original matrices. 
A direct calculation shows that 
\[ \left( \begin{array}{cccc}
1 & u_{12} & u_1 & u\\
0 & 1 & 0  & u_2\\
0 & 0 & 1 & u_{23}\\
0 & 0 & 0 & 1 \end{array} \right)^{-1} = \left( \begin{array}{cccc}
1 & -u_{12} & -u_1 & -u +u_{12}u_2+u_1 u_{23}\\
0 & 1 & 0  & -u_2\\
0 & 0 & 1 & -u_{23}\\
0 & 0 & 0 & 1 \end{array} \right).\]
It is again evident that the coordinates of the inverse matrix satisfy the differential equations stipulated above. So, the set of matrices is a differential algebraic group. 

The group is $0$-indecomposable. The reader should note that in the setting of differential algebraic groups, $0$-indecomposable simply means connected, that is, there are no definable subgroups of finite index. This group is also $1$-connected. 
The $2$-connected component is the subgroup of matrices of the form: 
\[ \left( \begin{array}{cccc}
1 & 0 & u_1 & u\\
0 & 1 & 0  & u_2\\
0 & 0 & 1 & 0\\
0 & 0 & 0 & 1 \end{array} \right)\]
The $3$-connected (in this case, strongly connected) component is the subgroup of matrices of the form: 
\[ \left( \begin{array}{cccc}
1 & 0 & 0 & u\\
0 & 1 & 0  & 0\\
0 & 0 & 1 & 0\\
0 & 0 & 0 & 1 \end{array} \right)\]
\end{exam}
Though much of the analysis of this paper essentially works in the case of $n$-indecomposability with $n \neq \tau(G),$ by relativizing the appropriate statements (see Proposition \ref{general}, for instance), there are important pieces which are not immediate. For instance, when seeking definability results for a family of $n$-indecomposable subsets, the above techniques are only useful when the subsets can be contained in a strongly connected subgroup of differential type $n.$

The indecomposability theorem of Berline and Lascar \cite{BerlineLascar} applies in the setting of superstable groups, so, specifically for groups definable in $DCF_{0,m}$. As we noted in the introduction, there is no known lower bound for Lascar rank in partial differential fields based on differential type and typical differential dimension. In fact, examples of \cite{SuerThesis} show that any such lower bound can not involve typical differential dimension (Suer constructs differential varieties of arbitrarily high typical differential dimension, differential type 1, and Lascar rank $\omega$). It is not currently known if there is an infinite transcendence degree strongly minimal type. One should note that such examples are present in the difference-differential context \cite{Medina}, but have yet to be discovered in the partial differential context. 

In certain special cases, when one has information about the Lascar rank for a differential algebraic group the indecomposability theorem given in \cite{BerlineLascar} may be applied to obtain similar conclusions as those given here. For most general classes of groups, this is not possible, however, there is at least one important exception. In the case that one considers almost simple differential algebraic group, $G/Z(G)$ is a simple differential algebraic group, to which the classification of \cite{CassidyClassification} applies. This classification gives a lower bound for the Lascar rank of $G$ in terms of the differential type.

\section{Acknowledgements}
I would like to thank Dave Marker for many helpful discussions, encouragement and support.  I would also like to thanks the referee for numerous corrections and insightful remarks. 

\bibliography{Research}{}
\bibliographystyle{plain}
\end{document}